\documentclass[11pt,reqno]{amsart}

\usepackage{amsmath,amssymb,amsthm} 
\usepackage{graphicx} 
\usepackage{dsfont}
\usepackage{array}
\usepackage{todonotes}

\usepackage{algpseudocode,algorithmicx,algorithm}

\algrenewcommand{\algorithmicrequire}{\textbf{Input:}}
\algrenewcommand{\algorithmicensure}{\textbf{Output:}}

\usepackage[margin=1in]{geometry}

\usepackage{epstopdf}

\newtheorem{theorem}{Theorem}[section]
\newtheorem{lemma}[theorem]{Lemma} 
\newtheorem{proposition}[theorem]{Proposition} 
\newtheorem{corollary}[theorem]{Corollary}

{\theoremstyle{definition}

\newtheorem{definition}{Definition}

 }

\newcommand{\abs}[1]{\left|{#1}\right|}

\newcommand{\set}[1]{\left\{{#1}\right\}} 
\newcommand{\setof}[2]{\left\{{#1}\,:\,{#2}\right\}}

\newcommand{\symd}{\bigtriangleup}

\newcommand{\G}{\mathcal{G}}

\newcommand{\lex}[2]{L(#1,#2)}
\newcommand{\J}{J}
\renewcommand{\j}{j}
\newcommand{\bigj}[1]{\j\left( #1 \right)}

\renewcommand{\S}[3]{R(#1,#2;#3)}
\renewcommand{\l}{\ell}
\renewcommand{\L}{L}
\newcommand{\N}{\mathbb{N}}

\newcommand{\hind}{H_\text{ind}}

\DeclareMathOperator{\Hom}{Hom}

\begin{document}
\title{Homomorphisms into loop-threshold graphs}
\author{Jonathan Cutler}
\address{Department of Mathematical Sciences, Montclair State University, Montclair, NJ 07043}
\email{jonathan.cutler@montclair.edu}
\thanks{The first author was sponsored by the National Security Agency under Grant H98230-15-1-0016.  The United States Government is authorized to reproduce and distribute reprints notwithstanding any copyright notation herein.}
\author{Nicholas Kass}
\address{Department of Mathematics, University of Nebraska-Lincoln, Lincoln, NE 68588}
\email{nkass@huskers.unl.edu}

\begin{abstract}	
	Many problems in extremal graph theory correspond to questions involving homomorphisms into a fixed image graph.  Recently, there has been interest in maximizing the number of homomorphisms from graphs with a fixed number of vertices and edges into small image graphs. For the image graph $\hind$, the graph on two adjacent vertices, one of which is looped, each homomorphism from $G$ to $\hind$ corresponds to an independent set in $G$. It follows from the Kruskal-Katona theorem that the number of homomorphisms to $\hind$ is maximized by the lex graph, whose edges form an initial segment of the lex order. 
	
	A \emph{loop-threshold graph} is a graph built recursively from a single vertex, which may be looped or unlooped, by successively adding either a looped dominating vertex or an unlooped isolated vertex at each stage.  Thus, the graph $\hind$ is a loop-threshold graph.  We survey known results for maximizing the number of homomorphisms into small loop-threshold image graphs.  The only extremal homomorphism problem with a loop-threshold image graph on at most three vertices not yet solved is $\hind\cup E_1$, where extremal graphs are the union of a lex graph and an empty graph.  The only question that remains is the size of the lex component of the extremal graph.  While we cannot give an exact answer for every number of vertices and edges, we establish the significance of and give a bound for $\ell(m)$, the number of vertices in the lex component of the extremal graph with $m$ edges and at least $m+1$ vertices.
\end{abstract}
\maketitle

\section{Introduction} 
\label{sec:intro}

Many problems in classical extremal graph theory can be stated in terms of graph homomorphisms.  A \emph{homomorphism} from a graph $G$ to a graph $H$ is a function $\phi:V(G)\to V(H)$ such that $\phi(x)\phi(y)\in E(H)$ whenever $xy\in E(G)$, i.e., $\phi$ is an edge-preserving map.  We let $\Hom(G,H)$ be the set of homomorphisms from $G$ to $H$ and $\hom(G,H)=\abs{\Hom(G,H)}$.  If we take the image graph to be $\hind$, a path on two vertices with one vertex looped (see Figure~\ref{fig:hind}), then elements of $\hom(G,\hind)$ correspond to independent sets in $G$.  This is because vertices mapped to the unlooped vertex of $\hind$ form an independent set and there are no other restrictions on the map.
\begin{figure}[h]
	\begin{center}
	\includegraphics{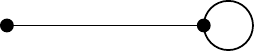}
	\end{center}
	\caption{The graph $\hind$.}\label{fig:hind}
\end{figure}
Another natural example arises from taking the image graph to be $K_q$, in which case elements of $\Hom(G,K_q)$ correspond to proper $q$-colorings of $G$.

An interesting class of problems arises by fixing an image graph $H$ and a class of graphs $\G$ and trying to determine which graphs $G\in \G$ maximize $\hom(G,H)$.  A lot of recent research has focused on this problem for the image graph $\hind$, which we saw above corresponds to independent sets, and so we let $i(G)=\hom(G,\hind)$.  Kahn \cite{K01} gave a bound for $i(G)$ when $G$ is an $r$-regular bipartite graph on $n$ vertices and Zhao \cite{Z10} extended this to general regular graphs.

\begin{theorem}[Kahn, Zhao]\label{theorem:kz}
	If $G$ is an $r$-regular graph on $n$ vertices, then
	\[
		i(G)\leq i(K_{r,r})^{n/2r}.
	\]
\end{theorem}

Galvin and Tetali \cite{GT04} extended the bipartite case of Theorem~\ref{theorem:kz} to general homomorphisms to any image graph.  Interestingly, the extension to nonbipartite graphs in the general case is not true.  (See \cite{G12} for more details.)  Also, the problem of maximizing the number of independent sets in graphs with given minimum degree has been well-studied (see, e.g., \cite{ACM, AM, CR14, CR15, EG, G11, GLS}).

In this paper, we will be interested in maximizing the number of homomorphisms over graphs of fixed order (number of vertices) and size (number of edges).  We let $\G(n,m)$ be the set of graphs with $n$ vertices and $m$ edges.  The problem of maximizing the number of independent sets over $\G(n,m)$ follows from the Kruskal-Katona \cite{K63,K68} theorem.  This is because independent sets in a graph $G$ correspond exactly to those sets that are not in the upper shadow of $G$, where $G$ is thought of as a set system on the vertex set.  Recall that the \emph{lexicographic order} on subsets of $[n]$ is defined by $A<B$ if $\min(A\symd B)\in A$.  Define the \emph{lex graph of order $n$ and size $m$}, denoted $L(n,m)$, to be the graph with vertex set $[n]=\set{1,2,\ldots,n}$ and edges consisting of the initial $m$ elements of $\binom{[n]}2=\setof{e\subset [n]}{\abs{e}=2}$ according to the lexicographic order.  We have the following, as was noted in \cite{CR11}.  

\begin{theorem}\label{thm:ind}
	If $G\in \G(n,m)$, then
	\[
		i(G)\leq i(L(n,m)),
	\]
	with equality if and only if $G\cong L(n,m)$.
\end{theorem}

Theorem~\ref{thm:ind} was also proved independently by Wood \cite{W}.  In fact, since the Kruskal-Katona theorem implies that the lex graph is extremal for independent sets of any fixed size and, as was noted in \cite{CR11a}, we have the following, where $i_t(G)$ is the number of independent sets of size $t$ in $G$.

\begin{theorem}\label{thm:levelind}
	If $G\in \G(n,m)$, then
	\[
		i_t(G)\leq i_t(L(n,m)),
	\]
	with equality if and only if $G\cong L(n,m)$.
\end{theorem}

This paper will begin with a survey of results about maximizing $\hom(G,H)$ where $G\in \G(n,m)$ and $H$ is a small fixed image graph.  In Section~\ref{sec:thres}, we introduce threshold graphs which will turn out to be extremal graphs for a set of image graphs consisting of what we call loop-threshold graphs.  In Section~\ref{sec:jhoms}, we begin our investigation into the one remaining open extremal problem for loop-threshold image graphs on at most three vertices.  Finally, in Section~\ref{sec:lm}, we give bounds on a function related to this extremal problem. 

	
\section{Threshold and loop-threshold graphs} 
\label{sec:thres}

Threshold graphs play a key role in the investigation of extremal problems related to graph homomorphisms.  There are many characterizations of threshold graphs (see, e.g., \cite{MP}), but the most useful one for our purposes is as follows.

\begin{definition}
	A graph $G$ is a \emph{threshold graph} if it can be constructed recursively from $K_1$ by successively adding either a dominating vertex or an isolated vertex.
\end{definition}

With this recursive definition in hand, we can observe that threshold graphs on $n$ vertices can by represented by binary sequences of length $n-1$, which we call the \emph{code} of the threshold graph.  We write $1$ for a dominating vertex and a $0$ for an isolated vertex.  As the code of the first vertex is irrelevant, we omit it from the code.  Following convention, we write the code from right to left.  Note that for a given code, ones are adjacent to all vertices to their right in the code and ones to their left.  Zeroes are adjacent only to ones to their left.  Note that threshold graphs have at most one nontrivial component.  

We often use superscripts to denote a string of the same symbol in the code of a threshold graphs so, $0^p1^q$ is the code with $q$ ones followed by $p$ zeroes (from right to left).  We note that any lex graph is a threshold graph.  In general, the code of a lex graph is of the form $1^p0^q1^a0^r$, where $p$, $q$, and $r$ are all nonnegative integers and $a$ may be either $0$ or $1$ (where we write $x^0$ for the empty string).

We say that a graph $G$ is \emph{$H$-extremal} for a graph $H$ (which may have loops) if 
\[
	\hom(G,H)=\max\setof{\hom(G',H)}{\text{$n(G')=n(G)$ and $e(G')=e(G)$}}.
\]
It turns out that for many image graphs $H$, one is able to prove that there is always an $H$-extremal graph that is threshold.  Such image graphs are called loop-threshold graphs.

\begin{definition}
	A graph with (perhaps) loops is \emph{loop-threshold} if it can be obtained from $K_1$, or $K_1$ with a loop, by successively adding an unlooped isolated vertex, or a looped dominating vertex.
\end{definition}

Once again, loop-threshold graphs on $n$ vertices can be associated with a binary code, except now the length of the code is $n$ (rather than $n-1$).  This is due to the fact that the code of the first vertex is now relevant since it determines whether that vertex is looped or not.  We note that $\hind$ is a loop-threshold graph with code $10$.  The following was proved in \cite{CR12}.

\begin{theorem}\label{thm:thrtothr}
	If $H$ is any loop-threshold graph and $n$ and $m$ are any non-negative integers with $0\leq m\leq \binom{n}2$, then there is a threshold graph $T\in \G(n,m)$ such that $T$ is $H$-extremal.
\end{theorem}

The purpose of the remainder of this section is to survey extremal results for $\hom(G,H)$ where $H$ is a loop-threshold graph on at most three vertices.  We begin by making a couple of trivial observations.  The first we state without proof.

\begin{proposition}
	If $q\geq 1$ is an integer and $H$ is any loop-threshold graph with code $0^q$, then $\hom(G,H)=0$ for any graph $G$ with at least one edge.
\end{proposition}

Another simple result holds for any image loop-threshold graph whose code is all ones.

\begin{proposition}
	If $p\geq 1$ is an integer and $H$ is a loop-threshold graph with code $1^p$ and $G$ is a graph with $n$ vertices, then $\hom(G,H)=p^n$.
\end{proposition}

There is one more simple case, namely when $H$ is a loop-threshold graph with code of the form $0^q 1^p$.

\begin{proposition}\label{prop:pm}
	If $p$ and $q$ are non-negative integers with $p\geq 1$ and $H$ is a loop-threshold graph with code $0^q 1^p$ and $G$ is a graph on $n$ vertices with $c$ isolated vertices, then $\hom(G,H)=(p+q)^c p^{n-c}$.
\end{proposition}

\begin{proof}
	Any of the $c$ isolates in $G$ can be mapped to any vertex of $H$, while the other $n-c$ non-isolates in $G$ can only be mapped to any of the $p$ non-isolates in $H$.
\end{proof}

If $H$ is a loop-threshold graph with code $0^q 1^p$, then Proposition~\ref{prop:pm} implies that $\hom(G,H)$ is a simple function of the number of isolates when $G\in \G(n,m)$.  To maximize $\hom(G,H)$ over $G\in \G(n,m)$, one simply has to maximize the number of isolates in $G$.  This can be done by a variety of graphs, including the colex\footnote{The \emph{colexicographic order} on subsets of $[n]$ sets $A<B$ if $\max(A\symd B)\in B$.  The \emph{colex graph with $n$ vertices and $m$ edges}, denoted $C(n,m)$, has vertex set $[n]$ and edge set consisting of the initial $m$ elements of $\binom{[n]}2$ in the colexicographic order.} graph $C(n,m)$.

The final general case we would like to present in this section is that of image loop-threshold graphs with code $1^p 0^q$.  We note that if $H$ is a loop-threshold graph with code $1^p 0^q$ and we consider a homomorphism $\phi\in \Hom(G,H)$ for some graph $G$, then the vertices of $G$ mapped to the vertices with code $0$ in $H$ form an independent set in $G$.  Thus, as one might expect, we can deduce the extremal question for maximizing homomorphisms into $H$ from earlier results about independent sets.

As one does often in statistical physics, we can weight homomorphisms to a graph $H$ with respect to a weighting function $\beta:V(H)\to [0,\infty)$.  Then, we ``count'' weighted homomorphisms from a graph $G$ to $H$ via the partition function:
\[
	\hom_{\beta}(G,H)=\sum_{\phi\in \Hom(G,H)} \prod_{v\in G} \beta\left(\phi(v)\right).
\]
Note that this reduces to counting homomorphisms if $\beta\equiv 1$, i.e., $\hom_{\beta}(G,H)=\hom(G,H)$.  We are interested in the case of weighted independent sets and so, if we label the vertices of $\hind$ as in Figure~\ref{fig:wtdhi}, we define the weight function $\beta^{\lambda}$ to be the weight function on $\hind$ defined by
\[
	\beta^{\lambda}(x)=\begin{cases}
		\lambda & \text{if $x=a$}\\
		1 & \text{if $x=b$}.
	\end{cases}
\]
\begin{figure}
	\includegraphics{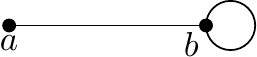}
	\caption{A labeled $\hind$}\label{fig:wtdhi}
\end{figure}
We see that, with this weighting of $\hind$, an independent set $I$ is assigned weight $\lambda^{\abs{I}}$.  But then the weighted homomorphism model introduced above yields the independence polynomial of the graph, i.e., $P_G(\lambda)=\hom_{\beta^{\lambda}}(G,H)$.  Since $P_G(\lambda)$ depends only on the number of independent sets of various sizes and the lex graph maximizes the number of independent sets of all sizes simultaneously, by Theorem~\ref{thm:levelind}, we have the following.

\begin{corollary}\label{cor:wtdind}
	If $G$ is a graph on $n$ vertices and $m$ edges and $\lambda>0$, then 
	\[
		P_G(\lambda)\leq P_{L(n,m)}(\lambda).
	\]
\end{corollary}

Finally, we can return to our original question regarding image loop-threshold graphs with codes of the form $1^p 0^q$.  We let $S^{\circ}(p,q)$ be the clique-looped split graph $K_p^{\circ}\vee E_q$, in which each vertex of $K_p$ is looped.\footnote{The \emph{join} of graphs $G$ and $H$, denoted $G\vee H$, has vertex set $V(G)\cup V(H)$ and edge set $E(G)\cup E(H)\cup \setof{xy}{x\in V(G), y\in V(H)}$.}  In other words, $S^{\circ}(p,q)$ is a loop-threshold graph with code $1^p 0^q$.  We have the following.

\begin{corollary}
	Let $p,q\geq 1$ and $G$ be a graph with $n$ vertices and $m$ edges.  Then 
\[
	\hom(G,S^{\circ}(p,q))\leq \hom(L(n,m),S^{\circ}(p,q)).
\]
\end{corollary}

\begin{proof}
	Set $\lambda=p/q$.  Then, we have $\hom(G,S^{\circ}(p,q))=q^n P_G(\lambda)$ and so the result follows from Corollary~\ref{cor:wtdind}.
\end{proof}

Armed with the above results, we present in Table~\ref{tab:smallthres} all loop-threshold graphs on three or less vertices.  All cases except for the last two are covered by one of the preceding results in this section.  The penultimate image graph $F$ (with code $101$) is known as ``the fox'' (or ``the wrench'' by Brightwell and Winkler \cite{BW99}) and the extremal question for this graph was studied by Cutler and Radcliffe \cite{CR12}.  They were able to determine a class of five threshold graphs which formed a minimal class of $F$-extremal graphs.

The final image graph (with code $010$), which we denote $J$, will be the main focus of this paper.  While it is relatively easy to determine the form of all $J$-extremal graphs, we will determine the size of the non-trivial component asymptotically up to a constant factor.

\begin{center}
\begin{table}[t]
\begin{tabular}{>{\centering\arraybackslash}c|>{\centering\arraybackslash}c|>{\centering\arraybackslash}c}\renewcommand{\arraystretch}{1.8}
	Code of $H$ & $H$ & $H$-extremal graph\\ \hline\hline
	$0$ & \includegraphics{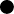} & Any\\
	$1$ & \includegraphics{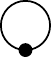} & Any \\ \hline
	$00$ & \includegraphics{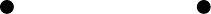} & Any \\
	$11$ & \includegraphics{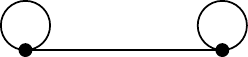} & Any \\  \hline
	$01$ & \includegraphics{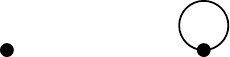} & $C(n,m)$\\ \hline
	$10$ & \includegraphics{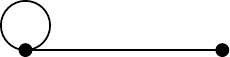} & $L(n,m)$\\ \hline
	$000$ & \includegraphics{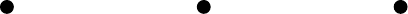} & Any \\
	$111$ & \includegraphics{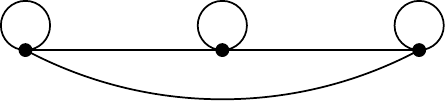} & Any \\ \hline
	$001$ & \includegraphics{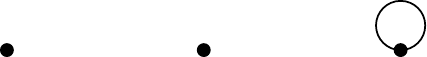} & $C(n,m)$ \\
	$011$ & \includegraphics{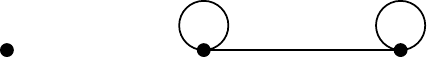} & $C(n,m)$ \\ \hline
	$100$ & \includegraphics{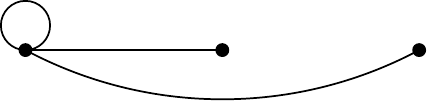} & $L(n,m)$ \\
	$110$ & \includegraphics{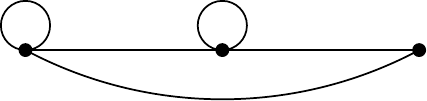} & $L(n,m)$\\ \hline
	$101$ & \includegraphics{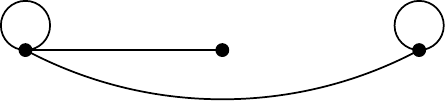} & See \cite{CR12}\\ \hline
	$010$ & \includegraphics{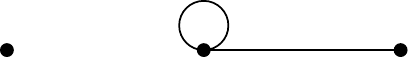} & See Section~\ref{sec:jhoms}
\end{tabular}
\caption{$H$-extremal graphs for all loop-threshold graphs $H$ on at most three vertices}\label{tab:smallthres}
\end{table}
\end{center}


\section{Homomorphisms into $\J$} 
\label{sec:jhoms}

As mentioned in the previous section, the only case remaining in finding a minimal set of $H$-extremal graphs, where $H$ is a loop-threshold graph on at most three vertices, is that of $H=\J$.  (See Figure~\ref{fig:j}.)  The code of $\J$ is $010$.
\begin{center}
	\begin{figure}[h]
		\includegraphics{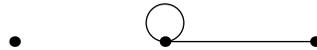}
		\caption{The loop-threshold graph $\J$.}\label{fig:j}
	\end{figure}
\end{center}
Given the fact that $J=E_1\cup \hind$, it is not surprising that the $J$-extremal graphs are comprised of a lex component with some number of added isolated vertices.  For a graph $G$, we let $j(G)=\hom(G,J)$.

\begin{theorem}\label{ext_composition}
Every $\J$-extremal threshold graph is the union of a lex graph and an empty graph.
\end{theorem}

\begin{proof} 
	The result is trivial if $m=0$.  If $m>0$ and $T$ is a threshold graph on $n$ vertices and $m$ edges then it contains a single nontrivial component, $C$.  Let $q=\abs{V(C)}$.  Since each vertex outside of that component can map to any vertex of $\J$ under a homomorphism in $\Hom(T,J)$, and every vertex contained in that component must map to the $\hind$ subgraph of $\J$, the number of homomorphisms from $T$ to $J$ is 
\[
\j (T)=3^{n-q}\,i(C),
\]
where $i(C)$ is the number of independent sets in $C$.  
As Theorem~\ref{thm:ind} establishes that $i(C)\leq i(L(q,m))$, with equality if and only if $C\cong \L(q,m)$, we have that
\begin{align*}
\j(T)&= 3^{n-q}\,i(C)\\
&\leq 3^{n-q}\,i(\L(q,m))\\
&=j(E_{n-q}\cup L(q,m)),
\end{align*}
where equality holds in the second step if and only if $T\cong E_{n-q}\cup L(q,m)$.
\end{proof}

As this construction of lex graphs with added isolated vertices occurs so frequently it is natural to define it as a function of the total number of vertices, the number of vertices in the lex component, and the number of edges.

\begin{definition}
For $n,q,m\in\N $ where $q\leq \min\left\{ n, m+1\right\}$ and $0\leq m\leq {q \choose 2}$, let $\S{n}{q}{m}$ be the graph on $n$ vertices and $m$ edges defined by the union 
\[
\S{n}{q}{m}=\lex{q}{m}\cup E_{n-q}.
\]

\end{definition}

The restriction $q\leq m+1$ assures that $\lex{q}{m}$ contained no isolated vertices.  We also note that
\begin{equation}\label{eqn:jnumber} 
\bigj{\S{n}{q}{m}}=3^{n-q} i \left( \lex{q}{m} \right).
\end{equation}


\section{An extremal question} 
\label{sec:lm}

In light of Theorem \ref{ext_composition}, finding a $\J$-extremal threshold graph on $n$ vertices and $m$ edges is thus reducible to finding $q$ within the bounds of $m\leq {q \choose 2}$ and $q\leq m+1$ such that $\S{n}{q}{m}$ is $\J$-extremal.  That is to say that the question of determining $\J$-extremality is  solely a matter of finding an appropriate number of vertices on which to form a lex component while leaving the rest as isolated vertices.  This is the problem with which the remainder of this paper will be concerned.  

Ideally, one would be able to find the size of the lex component in the $\J$-extremal graph for any $n$ and $m$.  In general however, this does not appear to be an amenable question.  For example, if one fixes $n$ and asks how the size of the lex component changes as $m$ increases, one might hope that this change is at least monotonic even if not easy to precisely quantify.  This, unfortunately, is not the case.  Figure~\ref{fig:comptwo} illustrates this when $n=50$ which appears to typify the behavior throughout all values of $n$ based on computer testing. 

\begin{center}
	\begin{figure}[h] 
		\includegraphics[width=6in]{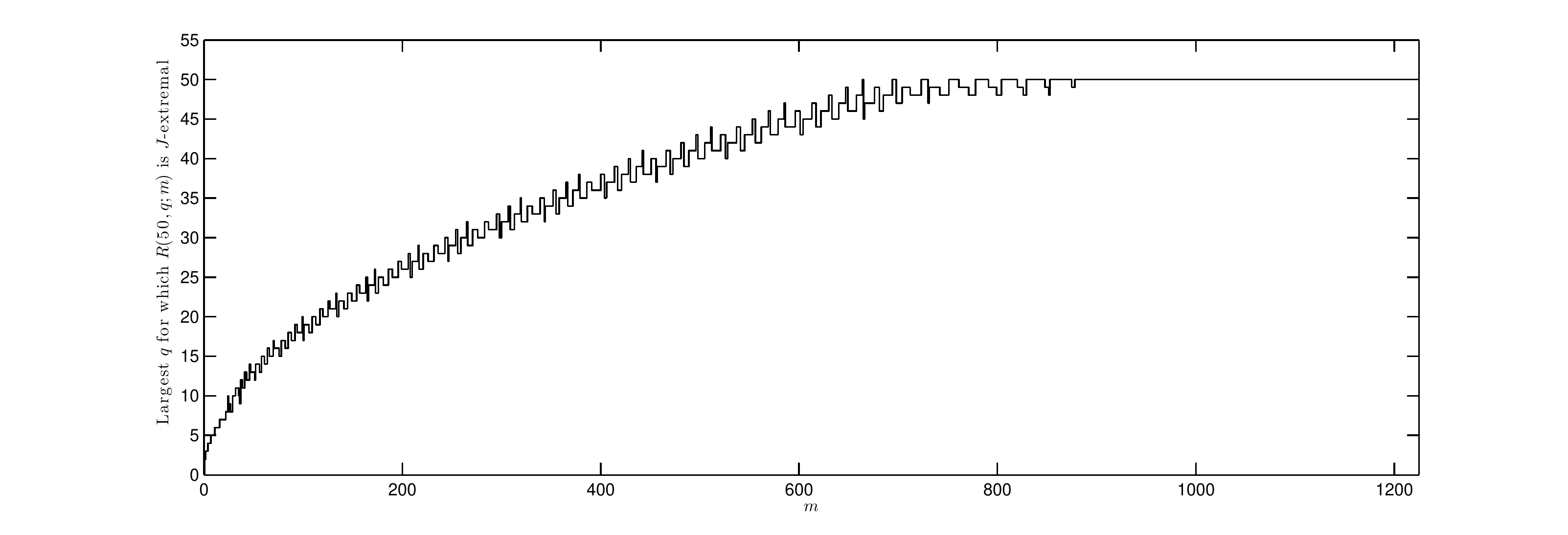}
		\caption{Computer testing showing non-monotone behavior of the lex component size in $J$-extremal graphs in the case where $n=50$. }\label{fig:comptwo}
	\end{figure}
\end{center}

Conversely, if one fixes the number of edges, $m$, the size of the lex component should remain fixed (at least in the limit) as $n$ increases to accord with equation (\ref{eqn:jnumber}) above since the size of the lex component cannot exceed $m+1$ vertices. In light of this non-monotonicity in terms of increasing values of $m$ and in order to capitalize on the relative stability of the problem in terms of increasing values of $n$, we instead consider the size of the lex component when $m$ is fixed. We define a function related to this behavior as follows.

\begin{definition}
For $m\in\N$, let $\l(m)$ be the maximum number of vertices in the non-trivial component of a $\J$-extremal threshold graph on $m+1$ vertices and $m$ edges. Equivalently,
\[
\l(m)=\max\setof{q}{\S{m+1}{q}{m}\text{ is }\j\text{-extremal}}
\]
\end{definition}

Here any uniqueness issues involving the extremal graphs are eliminated by the use of the maximum in our definition.  The choice of defining the composition of the extremal graph on exactly $m+1$ vertices relative to $m$ edges is not arbitrary; it corresponds to the largest possible size of the lex component in an extremal graph.  The following theorem proves the stronger claim that there is a single size of the lex component in a $J$-extremal graph on $m$ edges once $n$ is large enough.

\begin{theorem}
The graph $\S{n}{\l(m)}{m}$ on $n$ vertices and $m$ edges is $\J$-extremal for $n\geq\l (m)$.
\end{theorem}
\begin{proof}
Since we have required $n\geq \l(m)$ we may always add or remove $c=m+1-n$ isolated vertices to the graph $\S{n}{\l(m)}{m}$ to produce the graph $\S{m+1}{\l(m)}{m}$, in accordance with the sign of $c$. A simple counting argument on these isolates yields 
\[
\bigj{\S{n}{\l(m)}{m}}=3^{-c}\bigj{\S{m+1}{\l(m)}{m}},
\]
whereby the $J$-extremality of $\S{n}{\l(m)}{m}$ is now seen to be a consequence of the $\J$-extremality of $\S{m+1}{\l(m)}{m}$ which holds by the definition of $\l(m)$.

%
%
\end{proof}

Having established the significance of $\l(m)$ in the construction of $\J$-extremal graphs, we will concentrate on establishing an upper bound on $\l(m)$.  As the number of homomorphisms to $\J$ is so closely related to the number of independent sets in the lex component, the matter of bounding $\l(m)$ from above can be approached only once reasonable bounds are found for $i(\L(n,m))$.  Before embarking on these calculations, we pause to investigate some of the properties of the lex graph.  The lex graph $L(n,m)$ is related to (and sometimes isomorphic to) the \emph{split graph}.  We define the \emph{split graph $S(n,k)$} is the graph with $n$ vertices defined as the join of $K_k$ and $E_{n-k}$.  Thus, the number of edges in $S(n,k)$ is 
\[
	e(S(n,k))=\binom{k}2+k(n-k).
\]
If in the split graph $S(n,k)=K_k\vee E_{n-k}$, we label the vertices of the $K_k$ with $\set{1,2,\ldots,k}$ and those of the $E_{n-k}$ with $\set{k+1,k+2,\ldots,n}$, then we can use this to describe the structure of the lex graph.  If $m=\binom{k}2+k(n-k)+w$, we note that the lex graph with $n$ vertices and $m$ edges is
\[
	L(n,k,w):=L(n,m)=S(n,k)+\setof{(k+1)x}{k+2\leq x\leq k+w+2}.
\]
Given $n$ and $m$, we would like to determine $k$ and $w$ so that $L(n,m)=L(n,k,w)$.  Throughout this paper, we assume that $0\leq w\leq n-k-2$ so that there is no ambiguity about the value of $k$ and $w$ for a given $n$ and $m$.  Essentially, the lex graph $L(n,k,w)$ consists of the split graph $S(n,k)$ along with a star with $w$ edges inside the $E_{n-k}$.  (See Figure~\ref{fig:lex}.)  We note that in the threshold code of the lex graph, the center of the star of size $w$ in $L(n,k,w)$ corresponds to the lone one inside of the string of initial zeroes.  From this point forward, we refer to the two types of vertices in the lex graph as those from the \emph{complete part} and those from the \emph{empty part} (even though the empty part may contain a star).
	\begin{figure}
		\includegraphics{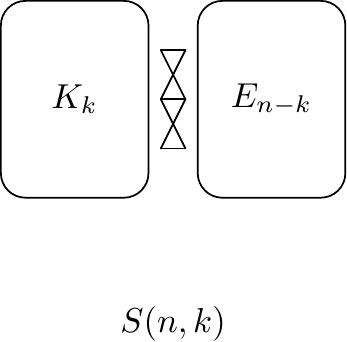}\hspace{.5in}\includegraphics{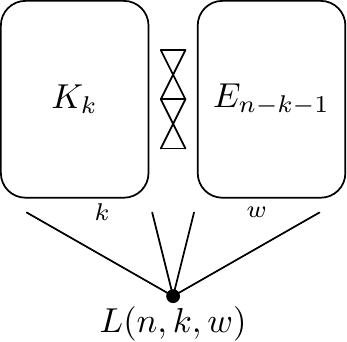}
		\caption{Schematics of the split graph and the lex graph}\label{fig:lex}
	\end{figure}
Given the lex graph $L(n,k,w)$, it is easy to determine the number of independent sets in it.  We state it as the following lemma without proof.
\begin{lemma}\label{lem:lexind}
	For the lex graph $L(n,k,w)$, we have
	\[
		i(L(n,k,w))=2^{n-k-1}+2^{n-k-w-1}+k.
	\]
\end{lemma}

\begin{proof}
	Note first that if any of the vertices in the complete part of the lex graph $L(n,k,w)$ are included in an independent set, then no other vertices can be in the independent set since these vertices are dominating.  There are $k$ such independent sets.  If, on the other hand, none of the vertices in the complete side are in an independent set $I$ and the center of the star of size $w$ in the empty part is also not in $I$, then there are $2^{n-k-1}$ such independent sets (including the empty set) since the rest of the vertices form an independent set.  If the center of the star of size $w$ in the empty set is in $I$, then none of the other vertices of the star can be in $I$, and so there are $2^{n-k-w-1}$ such $I$.  These account for all independent sets in $L(n,k,w)$, proving the lemma.
\end{proof}

By extension we may also determine $\j(\S{n}{n^\prime}{m})$ in a general form.  Noting that isolates can be mapped to any vertex of $J$ by an element of $\Hom(\S{n}{n^\prime}{m})$, we have by Lemma~\ref{lem:lexind},
\[
\j(\S{n}{n^\prime}{m})=3^{n-n^\prime}\left[k+2^{n^\prime-k-1}+2^{n^\prime-k-w-1}\right]
\]
where $k=k(n^\prime,m)$ and $w=w(n^\prime,m)$.

Now, we are finally ready to state and prove the main theorem of the paper, which gives an upper bound on $\l(m)$.

\begin{theorem}\label{thm:main}
For any $m\in \N$, we have
\[
\l(m) \leq \frac{5+\sqrt{9+24m}}2.
\]
\end{theorem}
\begin{proof}
	Fix $m\in \N$.  The proof will roughly go as follows: starting with $R(m+1,m+1;m)=L(m+1,m)$, we remove vertices from the lex component of the graph $R(m+1,n;m)$ (either one or two at a time), and create a $R(m+1,n';m)$ for some $n'<n$.  If we show that the number of homomorphisms into $J$ from this new graph is more than the earlier ones, then we can continue this process.  In the end, we will determine how long this process can be run, giving the bound in the statement of the theorem.  Note that in removing a vertex from the lex component of $R(m+1,n;m)$, we can think of forming the lex component of $R(m+1,n-1;m)$ by removing a vertex from the empty side of the lex component of $R(m+1,n;m)$ and then distributing the edges adjacent to this removed vertex.  If there is enough room to build a star of the right size in the empty part, then the complete part remains unchanged in size.  If not, then the complete part increases in size.  Our proof splits according to which of these occur.
	
	In order to describe the above process precisely, we need to introduce a bit of notation.  Fix some $n\leq m+1$.  Let $R=R(m+1,n;m)$.  In each step, we compare $j(R)$ to either $j(R')$ or $j(R')$ and $j(R'')$, where $R'=R(m+1,n-1;m)$ and $R''=R(m+1,n-2;m)$.  If we show that $j(R')>j(R)$ or $j(R'')>\max\set{j(R),j(R')}$, then we show that either $\ell(m)\leq n-1$ or $\ell(m)\leq n-2$, respectively, and then run through the process again.  We note that in comparing $j(R)$ with either $j(R')$ or $j(R'')$, we need only to compare the lex component of $R$, which is $L(n,m)$, with the lex component of $R'$ along with one isolated vertex, or the lex component of $R''$ along with two isolated vertices.  This is equivalent to dividing each of these by $3^{m+1-n}$.  To this end, let $G=L(n,m)$, $G'=L(n-1,m)\cup E_1$, and $G''=L(n-2,m)\cup E_2$.  Suppose that $k$ and $w$ are integers so that $L(n,m)=L(n,k,w)$.  Also, let $w',k',w'',k''$ be integers such that $L(n-1,m)=L(n-1,k',w')$ and $L(n-2,m)=L(n-2,k'',w'')$.  
	
	Note that if $n=m+1$, then $G=L(n,m)=K_{1,m}=L(n,1,0)$, so that $k=1$ and $w=0$, and $G'=L(n-1,m)\cup E_1=L(n-1,1,1)\cup E_1$, so that $k'=1$ and $w'=1$ (provided $m\geq 3$).  If $m$ is large, and we continue to remove vertices from the lex component, then the star inside of the empty part continues to grow while the empty part shrinks.  Eventually (after about $m/2$ steps), there will be no more room in the empty part, and so $k$ increases to two and $w$ resets to either $0$ or $1$.  The process speeds up now since the removal of any vertex leaves two edges that need to be placed in the star.  So far, each step of this process either leaves $k$ unchanged or increases it by one.  However, as the process continues, $k$ can increase by more than one.  Our process will work provided the change in $k$ is at most one at every step.  This barrier will give the bound in the theorem, as will be made precise below.
	
	Finally, we are ready to outline the cases of the argument.  They are split essentially so that we can determine $w'$, $k'$, $w''$, and $k''$ in terms of $n$ and $k$.  Throughout, we will keep track of conditions on $n$ and $k$ that are required.
	
\begin{description}
	\item[Case 1 ($k=k'$)] Note that in this case, there is room in the empty side of $G=L(n,m)$ to place the $k$ edges that are incident with the new isolated vertex in $G'$.
	\begin{description}
	\item[Subcase 1a ($w\geq 1$)] In this case, we compare $j(G)$ and $j(G')$.  By definition, we know that $w\leq n-k-2$ which implies that there is a vertex in the empty part of $L(n,k,w)$ that is not part of the star with $w$ edges.  Thus, its degree is $k$ and is only adjacent to the vertices of the $K_k$.  Deleting this vertex from $L(n,k,w)$ leaves $k$ edges that need to be ``placed'' in $L(n-1,m)$.  Since, by assumption, $k'=k$, we know that these must be able to fit inside of the empty part of $L(n-1,k,w')$ and form a larger star with $w+k$ edges now.  Thus, $w'=w+k$ and so we have the following.
	\begin{align*}
	j(G')-j(G)&=3j(L(n-1,k,w+k))-j(L(n,k,w))\\
		&=3\left(2^{(n-1)-k-1}+2^{(n-1)-k-(w+k)-1}+k\right)-\left(2^{n-k-1}+2^{n-k-w-1}+k\right)\\
		&=3\left(2^{n-k-2}+2^{n-2k-w-2}+k\right)-\left(2^{n-k-1}+2^{n-k-w-1}+k\right)\\
		&=2^{n-k-2}+3\cdot 2^{n-2k-w-2}-2^{n-k-w-1}+2k\\
		&=2^{n-2k-w-2}\left(2^{k+w}+3-2^{k+1}\right)+2k\\
		&> 0,
	\end{align*}
	where the last step follows from the assumption that $w\geq 1$.
	\item[Subcase 1b ($w=0$ and $k\geq 2$)] In this case, we must show that
	\[
		j(G'')>\max\set{j(G),j(G')}.
	\]
	In fact, we will show that $j(G')<j(G)<j(G'')$.  We do this in two steps, starting with showing that $j(G')<j(G)$.  Note that since $k'=k$ and $w=0$, we have that $w'=k$ provided that $n-k\geq k$ or $n\geq 2k$.  As above, we have $n'=n-1$.  Thus,
	\begin{align*}
		j(G)-j(G')&=j(L(n,k,0))-3j(L(n-1,k,k))\\
		&=2^{n-k-1}+2^{n-k-0-1}+k-3\left(2^{n-k-2}+2^{n-2k-2}+k\right)\\
		&=4\cdot 2^{n-k-2}+k-3\cdot 2^{n-k-2}-3\cdot 2^{n-2k-2}-3k\\
		&=2^{n-k-2}-3\cdot 2^{n-2k-2}-2k\\
		&=2^{n-2k-2}(2^k-3)-2k\\
		&>0,
	\end{align*}
	provided that $k\geq 2$ and $n>2k+\log_2 k+2$.
	
	It remains to show that $j(G'')>j(G)$.  We do need to make one assumption here.  Namely, we need to assume that there is enough ``room'' for the $2k$ displaced edges from vertices that become isolates in $G''$.  This occurs as long as $2k+1\leq n-k-2$ which is equivalent to $n\geq 3k+3$.  In this case, we have
	\begin{align*}
		j(G'')-j(G)&=9j(L(n-2,k,2k))-j(L(n,k,0))\\
		&\geq 9\left(2^{(n-2)-k-1}+2^{(n-2)-k-2k-1}+k\right)-2^{n-k-1}-2^{n-k-0-1}-k\\
		&=9\left(2^{n-k-3}+2^{n-3k-3}+k\right)-2^{n-k}-k\\
		&=2^{n-k-3}+9\cdot 2^{n-3k-3}+8k\\
		&>0.
	\end{align*}
	\item[Subcase 1c ($w=0$ and $k=1$)] In this case, we need only to check that $j(G')>j(G)$.  To this end, we have
	\begin{align*}
		j(G')-j(G)&=3j(L(n-1,1,1))-j(L(n,1,0))\\
		&=3\left(2^{n-3}+2^{n-4}+1\right)-2^{n-2}-2^{n-2}-1\\
		&=3\cdot 2^{n-3}+3\cdot 2^{n-4}+2-2^{n-1}\\
		&=2^{n-4}+2\\
		&>0.
	\end{align*}
	\end{description}
	\item[Case 2 ($k'=k+1$)] In this case, when we remove a vertex from the lex graph, we again need to sort the $k$ displaced edges to form a new lex graph on $n-1$ vertices.  Since, by assumption, our complete side of lex increases in size by one, the center of the star in the empty side of $G$ must get filled up with edges and move to the complete side.  This amounts to having the $w+k$ edges we would want to be incident to this vertex run out of room.  This occurs when $w+k\geq n-k-2$ or $w\geq n-2k-2$, the size of the empty side of $G'$.  Thus, we have $n'=n-1$, $k'=k+1$, and $w'=w+k-(n-k-2)=w+2k-n+2$.
	\begin{description}
	\item[Subcase 2a ($w=n-2k-2$ or $w=n-2k-1$)] In this case, we show that $j(G')>j(G)$.  Note that if $w=n-2k-2$, then $w'=0$ and if $w=n-2k-3$, then $w'=1$.  We have
	\begin{align*}
		j(G')-j(G)&=3j(L(n-1,k+1,w'))-j(L(n,k,w))\\
		&=3\left(2^{n-1-(k+1)-1}+2^{n-1-(k+1)-w'-1}+k+1\right)-2^{n-k-1}-2^{n-k-w-1}-k\\
		&=3\left(2^{n-k-3}+2^{n-k-w'-3}+k+1\right)-2^{n-k-1}-2^{n-k-w-1}-k\\
		&=-2^{n-k-3}+3\cdot 2^{n-k-w'-3}-2^{n-k-w-1}+2k+3\\
		&=2^{n-k-3}\left(-1+3\cdot 2^{-w'}-2^{-w+2}\right)+2k+3.
	\end{align*}
	If $w=n-2k-2$ and so $w'=0$, then
	\[
		-1+3\cdot 2^{-w'}-2^{-w+2}=-1+3\cdot 2^0 -2^{-n+2k+4}>0,
	\]
	provided that $n\geq 2k+4$.  If $w=n-2k-1$ and so $w'=1$, then
	\[
		-1+3\cdot 2^{-w'}-2^{-w+2}=-1+3\cdot 2^1 -2^{-n+2k+3}>0,
	\]
	provided that $n\geq 2k+1$.  In either case, we have $j(G')-j(G)>0$ as long as $2^{n-k-3}\geq 2k+3$ which is equivalent to $n\geq k+\log_2(2k+3)+3$.
	
	\item[Subcase 2b ($w\geq n-2k$)]  In this case, we will need to compare $j(G)$ with both $j(G')$ and $j(G'')$.  We again show that $j(G'')>j(G)>j(G')$.  We first show that $j(G)>j(G')$.  As in the previous subcase, we have $n'=n-1$, $k'=k-1$, and $w'=w+2k-n+2$.  In fact, we are just doing the calculation in the previous case, but multiplying by $-1$.  Thus, we have
	\begin{align*}
		j(G)-j(G')&=j(L(n,k,w))-3(j(L(n-1,k+1,w+2k-n+2)))\\
		&=2^{n-k-1}+2^{n-k-w-1}+k-3\left(2^{n-k-3}+2^{2n-3k-w-5}+k+1\right)\\
		&=2^{n-k-3}\left(1-3\cdot 2^{n-2k-w-2}+2^{-w+2}\right)-2k-3.
	\end{align*}
	Note that
	\[
		1-3\cdot 2^{n-2k-w-2}+2^{-w+2}>0 \qquad\text{if and only if}\qquad 2^w+4>3\cdot 2^{n-2k-2}.
	\]
	But the latter inequality certainly holds provided that $w\geq n-2k$, our assumption in this case.  Thus, $j(G)-j(G')>0$ provided that $n\geq k+\log_2(2k+3)+3$.
	
	Lastly, we need to show that $j(G'')>j(G)$.  Similar to Subcase 1b, we will assume that when we remove a vertex from the lex component of $G'$, there is enough room in the empty side of the lex component of $G''$ to fit the $k+1$ displaced edges.  This amounts to assuming that $w''=w'+k\leq n-k-5$ since the empty side of the lex component of $G''$ has size $n-k-3$.  But $w'=w+2k-n+2$ and $w\leq n-k-2$, so this assumption is met provided 
\[
	w''=w'+k=w+3k-n+2\leq (n-k-2)+3k-n+2=2k\leq n-k-2,
\]
i.e., $n\geq 3k+2$.  With this assumption, note that $n''=n-2$, $k''=k+1$, and $w''=w+3k-n+2$.  So, we have
\begin{align*}
	j(G'')-j(G)&=9j(L(n-2,k+1,w+3k-n+2))-j(L(n,k,w))\\
	&=9\left(2^{n-2-(k+1)-1}+2^{n-2-(k+1)-(w+3k-n+2)-1}+k+1\right)-2^{n-k-1}-2^{n-k-w-1}-k\\
	&=9\left(2^{n-k-4}+2^{2n-4k-w-6}+k+1\right)-2^{n-k-1}-2^{n-k-w-1}-k\\
	&=2^{n-k-4}+9\cdot 2^{2n-4k-w-6}-2^{n-k-w-1}+8k+9\\
	&>0,
\end{align*}
	provided that $w\geq 3$.  This is true provided $w\geq n-2k\geq 3$ or if $n\geq 2k+3$.
	\end{description}
\end{description}
	
	Note that in Subcase 1b, we assume that $n\geq 3k+3$.  Note that if this assumption holds, it implies that $n>2k+\log_2 k+2$, and $n\geq k+\log_2(2k+3)+3$, the assumptions in Subcase 1b, and Subcases 2a and b, respectively.  Also, we assume $n\geq 2k+3$ in Subcase 2a, which is also implied by $n\geq 3k+3$.
	
	All that remains is to show that the condition $n\geq 3k+3$ implies the bound in theorem.  Since we have assumed that $m\geq 1$, we know that $k\geq 1$, so this assumes that $n\geq 6$.  Since, initially, we have $n\leq m+1$, this means we have $m\geq 5$.  The cases $1\leq m\leq 4$ are easy to check.  In fact, in each of these cases, our bound is larger than $m+1$, the largest possible value of $\ell(m)$.  If $m\geq 5$, note that we can run the process provided that $n\geq 3k+3$ and each run removes either one or two vertices from the lex component.  We will calculate $n$ in terms of $m$ provided that $n=3k+4$.  (While our process may stop at $n=3k+3$ and not $3k+4$, our bound improves if we assume the former.)  In fact, it is easier to calculate $k$ and then use this to find $n$.  To this end, note that
	\[
		m=\binom{k}2+k(n-k)+w=\binom{k}2+k(2k+4)+w.
	\]
	Expanding this expression and solving for $k$, we get
	\[
		k=\frac{-3\pm\sqrt{9+24m-12w}}{6}.
	\]
	The negative root is contradictory to our assumption that $k\geq 0$, and so we see that
	\begin{align*}
		\ell(m)&\leq n\\
		&=3k+4\\
		&=\frac{-3+\sqrt{9+24m-12w}}2+4\\
		&\leq \frac{5+\sqrt{9+24m}}2,
	\end{align*}
	completing the proof of the theorem.
\end{proof}


\section{Future directions} 
\label{sec:future}

A question remains for $\l(m)$, namely, how good is the bound in Theorem~\ref{thm:main}.  Firstly, a trivial lower bound for $\l(m)$ is as follows.

\begin{lemma}
	For any $m\in \mathbb{N}$, we have
	\[
		\l(m)\geq \frac{1}{2}\left(1+\sqrt{1+8m}\right).
	\]
\end{lemma}
\begin{proof}
Assume not, so that $\l(m)<\frac{1}{2}(1+\sqrt{1+8m})$.  Then 
\[
{\l(m)\choose 2}<{{\frac{1}{2}(1+\sqrt{1+8m})}\choose 2}= m.
\]
Clearly this cannot be as it would imply a lex graph on $\l(m)$ vertices contains more than $\l(m)\choose 2$ edges. 
\end{proof}

So, summarizing our results, we have that there is a constant $c$ such that for any $m\in \N$,
\[
	1.4141\sqrt{m}\leq \l(m)\leq 2.4485\sqrt{m}+c.
\]
Through computer testing, it appears that the correct answer for $\l(m)$ is about $1.8\sqrt{m}$.  It would be nice to obtain an upper (and lower) bound that is closer to this.


\bibliographystyle{amsplain}
\bibliography{isoind}

\end{document}